\documentclass{amsproc}%
\usepackage{amssymb}
\usepackage{amsfonts}
\usepackage{amsmath}
\usepackage{graphicx}%
\theoremstyle{plain}

\newtheorem{corollary}{Corollary}

\newtheorem{example}{Example}

\newtheorem{lemma}{Lemma}

\newtheorem{proposition}{Proposition}

\newtheorem{theorem}{Theorem}
\numberwithin{equation}{section}
\begin{document}
\title{$\mathbb{C}$-orbit reflexive operators}
\author{Don Hadwin}
\address{University of New Hampshire}
\email{don@unh.edu}
\author{Ileana Ionascu }
\address{Philadelphia University}
\email{ionascui@philau.edu}
\author{Michael McHugh }
\address{Phillips Academy Andover}
\email{mmchugh@andover.edu}
\author{Hassan Yousefi}
\address{California State University Fullerton}
\email{hyousefi@fullerton.edu}

\begin{abstract}
We introduce the notion of $\mathbb{C}$-orbit reflexivity and study its
properties. An operator on a finite-dimensional space is $\mathbb{C}$-orbit
reflexive if and only if the two largest blocks in its Jordan form
corresponding to nonzero eigenvalues with the largest modulus differ in size
by at most one. Most of the proofs of our results in infinite dimensions are
obtained from purely algebraic results we obtain from linear-algebraic analogs
of $\mathbb{C}$-orbit reflexivity.

\end{abstract}
\maketitle

\bigskip

\section{Introduction\bigskip}

The term \emph{reflexive} for an algebra of operators was coined by P. R.
Halmos \cite{PRH}, but the first theorem about reflexivity was proved earlier
by D. Sarason \cite{Sa}. Since that time there have been many papers written
on the topic, and various notions of reflexivity (e.g., algebraic reflexivity
\cite{H1}, approximate reflexivity \cite{H2}, orbit reflexivity \cite{HNRR})
for linear subspaces, convex sets and other sets of operators, have been
studied extensively, e.g., \cite{Ar}, \cite{AP}, \cite{C}, \cite{Dedd},
\cite{DF}, \cite{Fil}, \cite{Lar}, \cite{LS}, \cite{P}, \cite{Sheh}. In
\cite{H3} a very general notion of reflexivity that contained many of these
notions as special cases. Orbit reflexivity was introduced and studied in
\cite{HNRR}, but it wasn't until over twenty years later that an operator was
constructed on a Hilbert space that is not orbit reflexive \cite{GR} (see also
\cite{MV} and \cite{Est}). John von Neumann's classic double commutant theorem
\cite{vN} can be viewed as the statement that every von Neumann algebra is
reflexive. In fact, many view reflexive algebras as nonselfadjoint analogues
of von Neumann algebras. Nest algebras, i.e., reflexive algebras whose lattice
of invariant subspaces is a chain, have received a great deal of attention
(see \cite{D}). W. Arveson \cite{Ar2} relates reflexivity to spectral
synthesis in commutative harmonic analysis and remarks that it is appropriate
to consider reflexivity questions as "noncommutative harmonic analysis".
Reflexivity appears in other guises, often in the form of a "local" property,
e.g., local derivations, local automorphisms, local multiplications \cite{HK},
\cite{HL}, \cite{Han},\cite{J}, \cite{Kad}, \cite{LS}, \cite{Sh}.

In this paper we introduce a new notion of reflexivity for operators,
$\mathbb{C}$\emph{-orbit reflexivity}, and we also define a linear-algebraic
analogue. This notion is related to the notion of orbit reflexivity
\cite{HNRR}. We first prove a number of results in the purely algebraic case,
and we use these to prove several results for operators on a normed space or a
Hilbert space. We also give an easy proof that every subnormal operator is
orbit reflexive. In finite-dimensions a characterization of reflexivity for a
single matrix was given \cite{DF} in terms of the Jordan form, i.e., for each
eigenvalue the largest two Jordan blocks differ in size by at most $1$. Every
matrix is orbit reflexive \cite{HNRR}. However, $\mathbb{C}$-orbit reflexivity
has a characterization in terms of the Jordan form that is similar to, but
quite different from, the one for reflexivity, i.e., among the nonzero
eigenvalues with maximum modulus the largest two Jordan blocks differ in size
by at most $1$.

Suppose $X$ is a normed space and $\mathcal{A}$ is an algebra of (bounded
linear) operators on $X$. A (closed linear) subspace $M$ of $X$ is
$\mathcal{A}$\emph{-invariant} if $A\left(  M\right)  \subseteq M$ for every
$A\in\mathcal{A}$. We let \textrm{Lat}$\mathcal{A}$ denote the set of all
invariant subspaces for $\mathcal{A}$, and we let \textrm{AlgLat}$\mathcal{A}$
denote the algebra of all operators that leave invariant every $\mathcal{A}%
$-invariant subspace. The algebra $\mathcal{A}$ is \emph{reflexive} if
$\mathcal{A}=$\textrm{AlgLat}$\mathcal{A}$. If the algebra $\mathcal{A}$
contains the identity operator $1$, then $S\in$\textrm{AlgLat}$\mathcal{A}$ if
and only if, for every $x\in X,$ $Sx$ is in the closure of $\mathcal{A}x$.
This characterization works equally well for linear subspace $\mathcal{S}$ of
$B\left(  X\right)  $ (the set of all operators on $X$), i.e., we define
\textrm{ref}$\mathcal{S}$ to be the set of all operators $A$ such that, for
every $x\in X$, we have $Ax$ is in the closure of $\mathcal{S}x$, and we say
that $\mathcal{S}$ is \textrm{reflexive} if $\mathcal{S}=\mathrm{ref}%
\mathcal{S}$. If we let $T$ be a single operator and let $\mathcal{S}%
=Orb\left(  T\right)  =\left\{  T^{n}:n\geq0\right\}  ,$ we apply the same
process to obtain the notion of orbit reflexivity. (Note that in this case
$\mathcal{S}$ is not a linear space.) We define OrbRef$\left(  T\right)  $ to
be the set of all operators $A$ such that, for every vector $x,$ we have $Ax$
is in the closure of Orb$\left(  T,x\right)  =$ Orb$\left(  T\right)  x$. We
say that $T$ is \emph{orbit reflexive} if OrbRef$\left(  T\right)  $ is the
closure of Orb$\left(  T\right)  $ in the strong operator topology (SOT). For
the next notion we allow powers and scalar multiples. We define, for the field
$\mathbb{F}\in\left\{  \mathbb{C},\mathbb{R}\right\}  $
\[
\mathbb{F}\text{\textrm{-}}Orb\left(  T\right)  =\left\{  \lambda T^{n}%
:n\in\mathbb{N},\lambda\in\mathbb{F}\right\}  ,
\]
and
\[
\mathbb{F}\text{\textrm{-}}Orb\left(  T,x\right)  =\left\{  \lambda
T^{n}x:n\geq0,\lambda\in\mathbb{F}\right\}  ,
\]%
\[
\mathbb{F}\text{\textrm{-}}\mathrm{OrbRef}\left(  T\right)  =\left\{  S\in
B\left(  H\right)  :Sx\in\mathbb{F}\text{\textrm{-}}Orb\left(  T,x\right)
^{-}\text{ for every }x\in H\right\}  .
\]
And we say that $T$ is $\mathbb{F}$\emph{-orbit reflexive} if $\mathbb{F}%
$-OrbRef$\left(  T\right)  $ is the strong operator closure of $\mathbb{F}%
$-Orb$\left(  T\right)  $.

\bigskip

\section{Algebraic Results}

Throughout this section $\mathbb{F}$ will denote an arbitrary field, $X$ will
denote a vector space over $\mathbb{F}$, and $\mathcal{L}\left(  X\right)  $
will denote the algebra of all linear transformations on $X$. If
$T\in\mathcal{L}\left(  X\right)  $, we define $\mathbb{F}$-OrbRef$_{0}\left(
T\right)  $ to be the set of all $S\in\mathcal{L}\left(  X\right)  $ such
that, for every $x\in X$, $Sx\in Orb\left(  T,x\right)  $. We say that $T$ is
algebraically $\mathbb{F}$\emph{-orbit reflexive} if $\mathbb{F}$%
-OrbRef$_{0}\left(  T\right)  =\mathbb{F}$-$Orb\left(  T\right)  $.

A transformation $T\in\mathcal{L}\left(  X\right)  $ is \emph{locally
nilpotent} if $X=\cup_{n\geq1}\ker\left(  T^{n}\right)  $. More generally $T$
is \emph{locally algebraic} if, for each $x\in X$, there is a nonzero
polynomial $p_{x}\in\mathbb{F}\left[  t\right]  $ such that $p_{x}\left(
T\right)  x=0$. If $p_{x}\left(  t\right)  $ is chosen to be monic with
minimal degree, we call $p_{x}$ a \emph{local polynomial} for $T$ at $x$.

\bigskip

\begin{theorem}
\label{locnil}Every locally nilpotent linear transformation on a vector space
$X$ over field $\mathbb{F}$ is algebraically $\mathbb{F}$-orbit reflexive.
Moreover, if $S\in\mathbb{F}$-OrbRef$_{0}\left(  T\right)  ,$ $f\in X,$
$\beta\in\mathbb{F}$, and $Sf=\beta T^{k}f\neq0,$ then $S=\beta T^{k}.$
\end{theorem}

\begin{proof}
Suppose first that $X$ is finite-dimensional and that $J_{n_{1}}\oplus
J_{n_{2}}\oplus\cdots\oplus J_{n_{k}}$ is the Jordan form for $T$ with
$n_{1}\geq n_{2}\geq\cdots\geq n_{k}\geq1.$ Suppose $S\in\mathbb{F}%
$-OrbRef$_{0}\left(  T\right)  .$ Suppose $e$ is in the domain of $J_{n_{1}}$
and $J_{n_{1}}^{n_{1}-1}e\neq0.$ We first assume that $Se\neq0.$ Then, since
$\left\{  e,Te,\ldots,T^{n_{1}-1}e\right\}  $ is linearly independent, there
is a unique $m,$ $0\leq m<n_{1}$ and a unique $\lambda\in\mathbb{F}$ such that
$Se=\lambda T^{m}e$. Suppose $g$ is in the domain of $J_{n_{2}}\oplus
\cdots\oplus J_{n_{k}}.$ Then there is an $\alpha\in\mathbb{F}$ and a $j,$
$1\leq j\leq n_{1}$ such that $Se+Sg=S\left(  e+g\right)  =\alpha T^{j}\left(
e+g\right)  =\alpha T^{j}e+\alpha T^{j}g$, and by projecting onto the domain
of $J_{n_{1}},$ we have $Se=\alpha T^{j}e$ which implies $j=m$ and
$\alpha=\lambda.$ Thus $Sg=\lambda T^{m}g$ for every $g$ in the domain of
$J_{n_{2}}\oplus\cdots\oplus J_{n_{k}}$. A similar argument, considering the
coefficient of $T^{m}e$ of $S\left(  e+g\right)  ,$ shows that $Sg=\lambda
T^{m}g$ if $g$ is any member of the linearly independent set $\left\{
Te,T^{2}e,\ldots,T^{n_{1}-1}e\right\}  .$ Hence $S=\lambda T^{m}$. Repeating
the same argument when $Se=0,$ gives $S=0.$ Hence $S\in$ $\mathbb{F}%
$-Orb$\left(  T\right)  $.

We now move to the general case. Suppose $T$ is locally nilpotent and
$S\in\mathbb{F}$-OrbRef$_{0}\left(  T\right)  .$ If $S=0$, we are done.
Suppose $f\in X$ and $Sf\neq0$. Choose $n\geq1$ so that $T^{n}f=0$ and
$T^{n-1}f\neq0.$ It follows that there is a $\beta\neq0$ in $\mathbb{F}$ and a
$k,$ $0\leq k<n$ such that $Sf=\beta T^{k}f\neq0$. Suppose $h\in X.$ Since $T$
is locally algebraic, it follows that $sp\left(  \left\{  T^{m}f:n\geq
0\right\}  \cup\left\{  T^{m}h:m\geq0\right\}  \right)  =M$ is a
finite-dimensional invariant subspace for $T,$ and, hence, for $S$. It follows
from our finite-dimensional case that there is an $m\geq0$ and a $\gamma
\in\mathbb{F}$ such that $Sx=\gamma T^{m}x$ for every $x\in M.$ In particular,
$0\neq Sf=\gamma T^{m}f,$ so $\gamma\neq0$ and $m<n$. We know that $\left\{
f,Tf,\ldots,T^{n-1}f\right\}  $ is linearly independent, so we have $m=k$ and
$\gamma=\beta.$ Thus $Sh=\beta T^{k}h.$ Since $h\in X$ was arbitrary, we have
$S=\beta T^{k}\in\mathbb{F}$-Orb$\left(  T\right)  $.
\end{proof}

\bigskip

For infinite fields the next theorem reduces the problem of algebraic
$\mathbb{F}$-orbit reflexivity to the case of locally algebraic
transformations. A key ingredient in the proof is an algebraic reflexivity
result from \cite{H1} that says if $\mathbb{F}$ is infinite and $T\in
\mathcal{L}\left(  X\right)  $ is not locally algebraic, then, whenever
$S\in\mathcal{L}\left(  X\right)  $ and for every $x\in X$ there is a
polynomial $p_{x}$ such that $Sx=p_{x}\left(  T\right)  x$, we must have
$S=p\left(  T\right)  $ for some polynomial $p$.

\begin{theorem}
\label{notlocalg} Suppose $X$ is a vector space over an infinite field
$\mathbb{F}$, and suppose $T\in\mathcal{L}\left(  X\right)  $ is not locally
algebraic. Then $T$ is algebraically $\mathbb{F}$-orbit reflexive.
\end{theorem}

\begin{proof}
Suppose $S\in\mathbb{F}$-OrbRef$_{0}\left(  T\right)  $. Then $Sx\in
\mathbb{F}$-Orb$\left(  T\right)  $ for every $x\in X$. It follows from
\cite{H1} that $T$ is algebraically reflexive, so we know there is a
polynomial $p\in\mathbb{F}\left[  t\right]  $ such that $S=p\left(  T\right)
.$ Since $T$ is not locally algebraic, there is a vector $e\in X$ such that
for every nonzero polynomial $q\in\mathbb{F}\left[  t\right]  ,$ we have
$q\left(  T\right)  e\neq0$. Since $S\in\mathbb{F}$-OrbRef$_{0}\left(
T\right)  ,$ we know that there is an $n\in\mathbb{N}$ and a $\lambda
\in\mathbb{F}$ such that $Se=\lambda T^{n}e.$ Hence $p\left(  t\right)
=\lambda t^{n},$ and thus $S\in\mathbb{F}$-Orb$\left(  T\right)  $.
\end{proof}

\bigskip

The following lemma dashes all hope, at least for some fields, that in finite
dimensions every transformation is algebraically $\mathbb{F}$-orbit reflexive.

\begin{lemma}
\label{fdlem}Suppose $\mathbb{F}$ is a field and $T$ is the linear
transformation on $\mathbb{F}^{2}$ defined by the matrix $T=\left(
\begin{array}
[c]{cc}%
1 & 1\\
0 & 1
\end{array}
\right)  .$ The following are equivalent:

\begin{enumerate}
\item $T$ is algebraically $\mathbb{F}$-orbit reflexive

\item $\mathbb{F}$ is not isomorphic to $\mathbb{Z}/p\mathbb{Z}$ for some
prime $p$.

\item Whenever $X$ is a vector space over $\mathbb{F}$ and $A,S$ are linear
transformations on $X$, $v\in X$ such that there is an $\beta\in\mathbb{F}$
and there are integers $k\geq0,m\geq2$ such that

\begin{enumerate}
\item $\left(  A-1\right)  v\neq0,$

\item $\left(  A-1\right)  ^{m}v=0,$

\item $S\in\mathbb{F}$-OrbRef$_{0}\left(  A\right)  $

\item $Sv=\beta A^{k}v,$
\end{enumerate}

then we must have $S\left(  A-1\right)  v=\beta A^{k}\left(  A-1\right)  v.$
\end{enumerate}
\end{lemma}

\begin{proof}
$\left(  1\right)  \Longrightarrow\left(  2\right)  $. Suppose $\left(
2\right)  $ is not true, and let $S=\left(
\begin{array}
[c]{cc}%
0 & 1\\
0 & 1
\end{array}
\right)  .$ Suppose $\left(
\begin{array}
[c]{c}%
x\\
y
\end{array}
\right)  \in\mathbb{F}^{2}$. If $y=0,$ then $S\left(
\begin{array}
[c]{c}%
x\\
y
\end{array}
\right)  =0=0T\left(
\begin{array}
[c]{c}%
x\\
y
\end{array}
\right)  $. Assume $y\neq0$, then there is a positive integer $m$ such that
\[
ym=1\text{ }\mathrm{mod}\text{ }p
\]
and let $n$ be a positive integer such that
\[
n=\left(  y-x\right)  m\text{ }\mathrm{mod}\text{ }p.
\]
Then we have
\[
T^{n}\left(
\begin{array}
[c]{c}%
x\\
y
\end{array}
\right)  =\left(
\begin{array}
[c]{cc}%
1 & n\\
0 & 1
\end{array}
\right)  \left(
\begin{array}
[c]{c}%
x\\
y
\end{array}
\right)  =\left(
\begin{array}
[c]{c}%
x+ny\\
y
\end{array}
\right)  =\left(
\begin{array}
[c]{c}%
x+\left(  y-x\right)  my\\
y
\end{array}
\right)  =S\left(
\begin{array}
[c]{c}%
x\\
y
\end{array}
\right)  .
\]
Hence $S\in\mathbb{F}$-OrbRef$_{0}\left(  T\right)  $, but $ST\neq TS,$ so
$\left(  1\right)  $ is not true.

$\left(  2\right)  \Longrightarrow\left(  1\right)  $. Suppose $\left(
2\right)  $ is true. Then we can choose $w\in\mathbb{F}$ so that
$w\notin\mathbb{Z}1.$ Suppose $S\in\mathbb{F}$-OrbRef$_{0}\left(  T\right)  .$
Then $S\left(
\begin{array}
[c]{c}%
1\\
0
\end{array}
\right)  =\left(
\begin{array}
[c]{c}%
a\\
0
\end{array}
\right)  $ for some $a\in\mathbb{F}$. Hence $S=\left(
\begin{array}
[c]{cc}%
a & b\\
0 & c
\end{array}
\right)  $ for some $b,c\in\mathbb{F}$. Suppose $c=0.$ Then
\[
\left(
\begin{array}
[c]{c}%
b\\
0
\end{array}
\right)  =S\left(
\begin{array}
[c]{c}%
0\\
1
\end{array}
\right)  =\alpha T^{n}\left(
\begin{array}
[c]{c}%
0\\
1
\end{array}
\right)  =\left(
\begin{array}
[c]{c}%
\alpha n\\
\alpha
\end{array}
\right)  ,
\]
so $\alpha=b=0.$ Now,
\[
\left(
\begin{array}
[c]{c}%
a\\
0
\end{array}
\right)  =S\left(
\begin{array}
[c]{c}%
1\\
1
\end{array}
\right)  =\alpha T^{n}\left(
\begin{array}
[c]{c}%
1\\
1
\end{array}
\right)  =\left(
\begin{array}
[c]{c}%
\alpha\left(  1+n\right) \\
\alpha
\end{array}
\right)  ,
\]
which implies $\alpha=0=a.$ Thus $c=0$ implies $S=0\in\mathbb{F}$-Orb$\left(
T\right)  .$ Hence we can assume $c\neq0.$

We now want to show $c=a$. Assume, via contradiction, that $c-a\neq0$. Then,
for some $\alpha\in\mathbb{F}$, and some integer $n\geq0$,%
\[
\left(
\begin{array}
[c]{c}%
a\frac{cw-b}{a-c}+b\\
c
\end{array}
\right)  =S\left(
\begin{array}
[c]{c}%
\frac{cw-b}{a-c}\\
1
\end{array}
\right)  =\alpha T^{n}\left(
\begin{array}
[c]{c}%
\frac{cw-b}{a-c}\\
1
\end{array}
\right)  =\alpha\left(
\begin{array}
[c]{c}%
\frac{cw-b}{a-c}+n\\
1
\end{array}
\right)  ,
\]
which implies that $\alpha=c\neq0$ and
\[
n=\frac{a}{c}\frac{cw-b}{a-c}+\frac{b}{c}-\frac{cw-b}{a-c}=w,
\]
which contradicts the choice of $w$. Thus $a=c$, so $S=\left(
\begin{array}
[c]{cc}%
a & b\\
0 & a
\end{array}
\right)  $ and $a\neq0$.

Finally we see that there is an $\alpha\in\mathbb{F}$ and an integer $n\geq0$
such that%
\[
\left(
\begin{array}
[c]{c}%
b\\
a
\end{array}
\right)  =S\left(
\begin{array}
[c]{c}%
0\\
1
\end{array}
\right)  =\alpha T^{n}\left(
\begin{array}
[c]{c}%
0\\
1
\end{array}
\right)  =\left(
\begin{array}
[c]{c}%
\alpha n\\
\alpha
\end{array}
\right)  ,
\]
which implies $b=na;$ whence, $S=aT^{n}$. Hence $\left(  1\right)  $ is true.

$\left(  3\right)  \Longrightarrow\left(  1\right)  $. Apply $\left(
3\right)  $ to the vector $v=\left(
\begin{array}
[c]{c}%
0\\
1
\end{array}
\right)  ,$ with $A=T$.

$\left(  2\right)  \Longrightarrow\left(  3\right)  $. Suppose $\left(
2\right)  $ is true. We can assume that $m$ is the smallest positive integer
for which $\left(  A-1\right)  ^{m}v=0$. It follows that $\left\{  \left(
A-1\right)  ^{s}v:0\leq s<m\right\}  $ is a linearly independent set whose
linear span $Y$ is an invariant subspace for $A.$ Similarly, the linear span
$M$ of $\left\{  \left(  A-1\right)  ^{s}v:s\geq2\right\}  $ is also an
invariant subspace for $A$. Since $S\in\mathbb{F}$-OrbRef$_{0}\left(
A\right)  $, we also have $S\left(  Y\right)  \subseteq Y$ and $S\left(
M\right)  \subseteq M$. Hence
\[
\hat{S}\left(  x+M\right)  =Sx+M\text{ and }\hat{A}\left(  x+M\right)  =Ax+M
\]
define linear transformations $\hat{S}$ and $\hat{A}$ on $Y/M$. It is easy to
see that $\hat{S}\in\mathbb{F}$-OrbRef$_{0}\left(  \hat{A}\right)  $ and the
matrix for $\hat{A}$ with respect to the basis $\left\{  \left(  A-1\right)
u+M,u+M\right\}  $ is $\left(
\begin{array}
[c]{cc}%
1 & 1\\
0 & 1
\end{array}
\right)  $ . Thus, by $\left(  1\right)  ,$ we know that there is a $\gamma
\in\mathbb{F}$ and an integer $t\geq0$ such that
\[
\hat{S}=\gamma\hat{A}^{t}.
\]
Thus,
\[
\hat{S}\left(  u+M\right)  =\gamma\left(  u+M\right)  +\gamma t\left(  \left(
A-1\right)  u+M\right)  .
\]
But $Su=\beta A^{k}u=\beta u+\beta k\left(  A-1\right)  u+h$ with $h\in M$.
Therefore,
\[
\hat{S}\left(  u+M\right)  =\beta\left(  u+M\right)  +\beta k\left(  \left(
A-1\right)  u+M\right)  ,
\]
which implies $\gamma=\beta$ and $\gamma t=\beta k$. On the other hand,
\[
\hat{S}\left(  \left(  A-1\right)  u+M\right)  =\gamma\left(  A-1\right)  u+M
\]
and, for some $\alpha$ and some $n$,%
\[
S\left(  A-1\right)  u=\alpha A^{n}\left(  A-1\right)  u=\alpha\left(
A-1\right)  u+h
\]
with $h\in M,$ which implies $\alpha=\gamma$. Hence if $\beta=0$, we have
$Su=0=S\left(  A-1\right)  u$. If $\beta\neq0,$ then $t=k$ and $\beta
=\gamma=\alpha$. However, there exist $\eta$ and $q$ such that
\[
S\left(  u+\left(  A-1\right)  u\right)  =\eta A^{q}\left(  A-1\right)  u=\eta
u+\eta\left(  q+1\right)  \left(  A-1\right)  u+g
\]
with $g\in M,$ which implies
\[
\hat{S}\left(  u+\left(  A-1\right)  u+M\right)  =\eta\left(  u+M\right)
+\eta\left(  q+1\right)  \left(  \left(  A-1\right)  u+M\right)  .
\]
Comparing this with%
\[
\hat{S}\left(  u+\left(  A-1\right)  u+M\right)  =\gamma\hat{A}^{t}\left(
u+\left(  A-1\right)  u+M\right)  =\beta\left(  u+M\right)  +\beta\left(
k+1\right)  \left(  \left(  A-1\right)  u+M\right)  ,
\]
we see that $\eta=\beta$ and $q=k.$ Hence,%
\[
S\left(  A-1\right)  u=S\left(  u+\left(  A-1\right)  u\right)  -Su=\beta
A^{k}\left(  u+\left(  A-1\right)  u\right)  -\beta A^{k}u=\beta A^{k}\left(
A-1\right)  u.
\]

\end{proof}

\bigskip

\begin{example}
\label{counter}Let $\omega_{n}=e^{\frac{2\pi i}{n}}$ for $n\geq1$. Let $Y$ be
a vector space over $\mathbb{C}$ with a basis $\left\{  e_{1},e_{2}%
,\ldots\right\}  $. Define linear transformations $A,B:Y\rightarrow Y$ by%
\[
Ae_{1}=0,Ae_{n+1}=e_{n}\text{ for }n\geq1,
\]
and%
\[
Be_{n}=\omega_{n}e_{n}\text{ for }n\geq1.
\]
Let $T=A\oplus B$ acting on $X=Y\oplus Y$. Let $S=0\oplus1$ acting on $X$.
Suppose $x\in X.$ Then there is a positive integer $n$ and scalars
$a_{1},b_{1},\ldots,a_{n},b_{n}\in\mathbb{C}$ such that
\[
x=\sum_{k=1}^{n}a_{k}e_{k}\oplus\sum_{k=1}^{n}b_{k}e_{k}.
\]
Then $Sx=T^{n!}x.$ However, there is no integer $N$ and scalar $\alpha$ such
that $S=\alpha T^{N}$, since $\alpha T^{N}\left(  e_{N+1}\oplus e_{N+1}%
\right)  \neq S\left(  e_{N+1}\oplus e_{N+1}\right)  $. Hence $T$ is neither
algebraically orbit reflexive nor algebraically $\mathbb{C}$-orbit reflexive.
\end{example}

\bigskip

The preceding example makes us look at the \emph{strict topology} on
$\mathcal{L}\left(  X\right)  ,$ where a basic neighborhood of a
transformation $T$ is given by a finite subset $E$ of $X$, defined by%
\[
U\left(  T,E\right)  =\left\{  A\in\mathcal{L}\left(  X\right)  :Ax=Tx\text{
for all }x\in E\right\}  .
\]
It is easy to show that if $S$ and $T$ are as in the preceding example, then
$T^{n}\rightarrow S$ in the strict topology. It is also easy to see that
$\mathbb{F}$-OrbRef$_{0}\left(  T\right)  $ and OrbRef$_{0}\left(  T\right)  $
are closed in the strict topology. It is natural to define a linear
transformation $T$ on a vector space $X$ over a field $\mathbb{F}$ to be
\emph{strictly algebraically }$\mathbb{F}$\emph{-orbit reflexive} if
$\mathbb{F}$-OrbRef$_{0}\left(  T\right)  $ is the strict closure of
$\mathbb{C}$-Orb$\left(  T\right)  $.

\bigskip

\bigskip

\begin{theorem}
\label{findim}Suppose $X$ is a finite-dimensional vector space over a field
$\mathbb{F}$ not isomorphic to $\mathbb{Z}/p\mathbb{Z}$ for some prime $p$.
Then every linear transformation on $X$ whose minimal polynomial splits over
$\mathbb{F}$ is algebraically $\mathbb{F}$-orbit reflexive.
\end{theorem}

\begin{proof}
Since the minimal polynomial for $T$ splits over $\mathbb{F}$, we can assume
$T$ has a Jordan canonical form. Moreover, if $T$ is nilpotent, then, by
Theorem \ref{locnil}, $T$ is algebraically $\mathbb{F}$-orbit reflexive. Thus
we can assume that $T$ has at least one nonzero eigenvalue $\lambda$ with
largest Jordan block of size $m$, which we can assume is $1.$ We can write
\[
T=\left(  1+J_{m}\right)  \oplus%
{\displaystyle\sum\nolimits_{1\leq i\leq s}^{\oplus}}
\left(  \alpha_{i}+J_{m_{i}}\right)  \oplus%
{\displaystyle\sum\nolimits_{1\leq j\leq t}^{\oplus}}
J_{n_{j}}%
\]
with $\alpha_{1},\ldots,\alpha_{s}$ nonzero and $n_{1}\geq\cdots\geq n_{t}$,
$m\geq m_{1}\geq\cdots\geq m_{s}$. Suppose $S\in\mathbb{F}$-OrbRef$_{0}\left(
T\right)  .$ First suppose there is a nonzero $f$ in the domain of $%
{\displaystyle\sum\nolimits_{1\leq j\leq t}^{\oplus}}
J_{n_{j}}$ and a $\beta\in\mathbb{F}$ and an integer $k\geq0$ such that
$Sf=\beta T^{k}f$. Then, by Theorem \ref{locnil}, this uniquely defines $k$
and $\beta$ and uniquely defines $S=\beta T^{k}$ on the domain of the
nilpotent transformation $%
{\displaystyle\sum\nolimits_{1\leq j\leq t}^{\oplus}}
J_{n_{j}}$. If $x$ is in the domain of $\left(  1+J_{m}\right)  \oplus%
{\displaystyle\sum\nolimits_{1\leq i\leq s}^{\oplus}}
\left(  \alpha_{i}+J_{m_{i}}\right)  $, then there is an integer $n$ and a
scalar $\gamma$ such that $S\left(  x+f\right)  =\gamma T^{n}\left(
x+f\right)  .$ But $S\left(  x+f\right)  =Sx+Sf$ and $\gamma T^{n}\left(
x+f\right)  =\gamma T^{n}x+\gamma T^{n}f.$ It follows that $Sx=\gamma T^{n}x$
and $Sf=\gamma T^{n}$, which implies $\gamma=\beta$ and $n=k.$ Hence,
$Sx=\beta T^{k}x.$ Therefore $\ker\left(  S-\beta T^{k}\right)  $ contains the
domains of both $\left(  1+J_{m}\right)  \oplus%
{\displaystyle\sum\nolimits_{1\leq i\leq s}^{\oplus}}
\left(  \alpha_{i}+J_{m_{i}}\right)  $ and $%
{\displaystyle\sum\nolimits_{1\leq i\leq t}^{\oplus}}
J_{n_{j}}$, which implies $S=\beta T^{k}$.

We now consider the case in which $S=0$ on the domain of $%
{\displaystyle\sum\nolimits_{1\leq j\leq t}^{\oplus}}
J_{n_{j}}$. Choose a vector $g$ in the domain of $J_{n_{1}}$such that
$J_{n_{1}}^{n_{1}-1}g\neq0$. If $m=1,$ then $m_{1}=\cdots=m_{s}=1$ and
$\left(  1+J_{m}\right)  \oplus%
{\displaystyle\sum\nolimits_{1\leq i\leq s}^{\oplus}}
\left(  \alpha_{i}+J_{m_{i}}\right)  $ is a diagonal matrix with eigenvectors
$u,u_{1},\ldots,u_{s}$ and there is a scalar $\eta$ and an integer $r$ such
that%
\[
S\left(  u+u_{1}+\cdots+u_{s}+g\right)  =\eta T^{r}\left(  u+u_{1}%
+\cdots+u_{s}+g\right)  ,
\]
which implies that $Su=\eta T^{r}u,Su_{i}=\eta T^{r}u_{i}$ $\left(  1\leq
i\leq s\right)  ,$ and $0=Sg=\eta T^{r}g,$ which implies $\eta T^{r}=0$ on the
domain of $%
{\displaystyle\sum\nolimits_{1\leq j\leq t}^{\oplus}}
J_{n_{j}}$, and, hence, $S=\eta T^{r}$.

We are left with the case where $S=0$ on the domain of $%
{\displaystyle\sum\nolimits_{1\leq j\leq t}^{\oplus}}
J_{n_{j}}$ and $m\geq2$. Let $\left\{  e_{1},\ldots,e_{m}\right\}  $ be the
basis shifted by $J_{m}$. Then there is a scalar $\rho$ and an integer
$N\geq0$ such that $Se_{1}=\rho T^{N}e_{1}$. It follows from part $\left(
3\right)  $ of Lemma \ref{fdlem} that $Se_{j}=\rho T^{N}e_{j}$ for $1\leq
j\leq m$. Suppose $y$ is in the domain of $%
{\displaystyle\sum\nolimits_{1\leq i\leq s}^{\oplus}}
\left(  \alpha_{i}+J_{m_{i}}\right)  \oplus%
{\displaystyle\sum\nolimits_{1\leq j\leq t}^{\oplus}}
J_{n_{j}}$, then there is a scalar $\alpha$ and an integer $d\geq0$ such that
\[
Se_{1}+Sy=S\left(  e_{1}+y\right)  =\alpha T^{d}\left(  e_{1}+y\right)
=\alpha T^{d}e_{1}+\alpha T^{d}y,
\]
and it follows that
\[
\rho T^{N}e_{1}=\alpha T^{d}e_{1}\text{ and }Sy=\alpha T^{d}y.
\]
However, the representation of $\rho T^{N}e_{1}$ with respect to the basis
$\left\{  e_{1},\ldots,e_{m}\right\}  $ is $\rho e_{1}+\rho Ne_{2}+\cdots$ and
the expansion for $\alpha T^{d}e_{1}$ is $\alpha e_{1}+\alpha de_{2}+\cdots,$
so $\alpha=\rho,$ and if $\alpha=\rho=0,$ then $Sy=0=\rho T^{N}y,$ and if
$\rho\neq0,$ then $d=N$ and $Sy=\rho T^{N}y.$ Hence $S=\rho T^{N}$.
\end{proof}

\bigskip

\begin{corollary}
If $X$ is a finite-dimensional vector space on an algebraically closed field
$\mathbb{F}$, then every linear transformation on $X$ is $\mathbb{F}%
$-algebraically reflexive. \bigskip
\end{corollary}

Recall from ring theory that if $\mathcal{R}$ is a principle ideal domain and
$M$ is an $\mathcal{R}$-module and $0\neq r\in\mathcal{R}$ and $rM=\left\{
0\right\}  ,$ then $M$ is a direct sum of cyclic $\mathcal{R}$-modules;
Applying this fact to $\mathcal{R}=\mathbb{F}\left[  t\right]  $, we get that
any algebraic linear transformation on a vector space is a direct sum of
transformations on finite-dimensional subspaces, and therefore has a Jordan
form when the minimal polynomial splits over $\mathbb{F}$. (See \cite{Kap} for
details.) This gives us the following corollary.

\bigskip

\begin{corollary}
\label{algTrans}Suppose $X$ is a vector space over a field $\mathbb{F}$ not
isomorphic to $\mathbb{Z}/p\mathbb{Z}$ for some prime $p$. Then every
algebraic linear transformation on $X$ whose minimal polynomial splits over
$\mathbb{F}$ is algebraically $\mathbb{F}$-orbit reflexive.
\end{corollary}

\bigskip

The next corollary follows from the technique in the last paragraph of the
proof of Theorem \ref{findim}. Recall from the beginning of Section 2 that if
$T$ is locally algebraic and $x$ is a vector, then the \emph{local minimal
polynomial} for $T$ at $x$ is the unique monic polynomial $p\left(  t\right)
$ of minimal degree for which $p\left(  T\right)  x=0$.

\begin{corollary}
Suppose $X$ is a vector space over a field $\mathbb{F}$ that is not isomorphic
to $\mathbb{Z}/p\mathbb{Z}$ for some prime $p$, and suppose $T$ is a locally
algebraic linear transformation on $X$ whose local minimal polynomials split
over $\mathbb{F}$. If there is a nonzero $\lambda\in\mathbb{F}$ such that
$\ker\left(  T-\lambda\right)  \neq\ker\left(  T-\lambda\right)  ^{2},$ then
$T$ is algebraically $\mathbb{F}$-orbit reflexive.
\end{corollary}

The next corollary follows from the fact that if $T$ is a locally algebraic
linear transformation and $E$ is any finite set of vectors, then there is a
finite-dimensional invariant subspace $M$ for $T$ that contains $E$.

\begin{corollary}
\label{lafd}Suppose $X$ is a vector space over a field $\mathbb{F}$ that is
not isomorphic to $\mathbb{Z}/p\mathbb{Z}$ for some prime $p$, and suppose $T$
is a locally algebraic linear transformation on $X$ whose local minimal
polynomials split over $\mathbb{F}$. Then $T$ is strictly algebraically
$\mathbb{F}$-orbit reflexive.
\end{corollary}

\begin{theorem}
If $\mathbb{F}$ is an algebraically closed field, then every linear
transformation on a vector space over $\mathbb{F}$ is strictly algebraically
$\mathbb{F}$-orbit reflexive.
\end{theorem}

\begin{proof}
Since $\mathbb{F}$ is algebraically closed, we know $\mathbb{F}$ is infinite
and is therefore not isomorphic to $\mathbb{Z}/p\mathbb{Z}$ for some prime
number $p$. Suppose $X$ is a vector space and $T$ is a linear transformation
on $X$. If $T$ is not locally algebraic, then $T$ is algebraically
$\mathbb{F}$-orbit reflexive. If $T$ is locally algebraic, then, by Corollary
\ref{lafd}, $T$ is strictly algebraically $\mathbb{F}$-orbit reflexive.
\end{proof}

\bigskip\ 

\section{$\mathbb{F}$-orbit reflexivity with $\mathbb{F}=\mathbb{C}$ or
$\mathbb{F=R}$\bigskip}

\begin{proposition}
Every normal operator is $\mathbb{C}$-orbit reflexive.
\end{proposition}

\begin{proof}
This is an immediate consequence of \cite[Proposition 3]{HNRR}.
\end{proof}

\bigskip

The next two results are consequences of Theorem \ref{locnil}.

\begin{theorem}
\label{unionker}Suppose $T$ is a bounded linear operator on a normed space $X$
over the field $\mathbb{F}\in\left\{  \mathbb{R}\emph{,}\text{ }%
\mathbb{C}\right\}  $ such that $\cup_{n=1}^{\infty}\ker\left(  T^{n}\right)
$ is dense in $X$. Then $T$ is $\mathbb{F}$-orbit reflexive and $\mathbb{F}%
$-Orb$\left(  T\right)  $ is SOT-closed. Moreover, if $S\in\mathbb{F}%
$-OrbRef$\left(  T\right)  $, $x\in X,$ $\beta\in\mathbb{F}$, $k\geq0,$ and
$Sx=\beta T^{k}x\neq0$, then $S=\beta T^{k}$.
\end{theorem}

\begin{proof}
Suppose $S\in\mathbb{F}$-OrbRef$\left(  T\right)  $ and let $M=\cup
_{n=1}^{\infty}\ker\left(  T^{n}\right)  $. It is clear that $S\left(
M\right)  \subseteq M$ and $T\left(  M\right)  \subseteq M$ and $S|_{M}%
\in\mathbb{F}$-OrbRef$\left(  T|M\right)  $. But $T|M$ is locally nilpotent,
and if $x\in M$ and $T^{n}x=0$, then
\[
\mathbb{F}\text{-}Orb\left(  T\right)  =\cup_{k=0}^{n}\mathbb{F}T^{k}x
\]
is norm closed. Hence, $\mathbb{F}$-OrbRef$\left(  T|M\right)  =\mathbb{F}%
$-OrbRef$_{0}\left(  T|M\right)  ,$ which, by Theorem \ref{locnil} is
$\mathbb{F}$-Orb$\left(  T\right)  $. Hence there is a $\lambda\in\mathbb{F}$
and an $n\geq0$ such that $S|M=\lambda T^{n}|M$. However, $M$ is dense in $X,$
so $S=\lambda T^{n}\in\mathbb{F}$-Orb$\left(  T\right)  $.
\end{proof}

\bigskip

The preceding theorem implies a stronger version of itself.

\begin{corollary}
Suppose $X$ is a normed space over $\mathbb{F}\in\left\{  \mathbb{R}%
,\mathbb{C}\right\}  $, $T\in B\left(  X\right)  ,$ and there is a
decreasingly directed family $\left\{  X_{\lambda}:\lambda\in\Lambda\right\}
$ of $T$-invariant closed linear subspaces such that
\end{corollary}

\begin{enumerate}
\item for every $\lambda\in\Lambda$, $\cup_{n=0}^{\infty}\left(  T^{n}\right)
^{-1}\left(  X_{\lambda}\right)  $ is dense in $X$, and

\item $\cap_{\lambda\in\Lambda}X_{\lambda}=\left\{  0\right\}  .$
\end{enumerate}

Then $T$ is $\mathbb{F}$-orbit reflexive.

\begin{proof}
Suppose $S\in\mathbb{F}$-OrbRef$\left(  T\right)  $ and $S\neq0$. Choose $e\in
X$ such that $Se\neq0$. It follows from $\left(  2\right)  $ that both
$\left(  1\right)  $ and $\left(  2\right)  $ remain true if we consider only
those $X_{\lambda}$ that contain neither $e$ nor $Se$. Since $T\left(
X_{\lambda}\right)  \subseteq X_{\lambda}$, $\hat{T}_{\lambda}\left(
x+X_{\lambda}\right)  =Tx+X_{\lambda}$ defines a bounded linear operator
$\hat{T}_{\lambda}$ on $X/X_{\lambda}$. Condition $\left(  1\right)  $ implies
that $\cup_{n=1}^{\infty}\ker\left(  \hat{T}_{\lambda}^{n}\right)  $ is dense
in $X/X_{\lambda};$ whence $\hat{T}_{\lambda}$ is $\mathbb{F}$-orbit
reflexive. However, $S\in\mathbb{F}$-OrbRef$\left(  T\right)  $ implies that
$S\left(  X_{\lambda}\right)  \subseteq X_{\lambda},$ so $\hat{S}_{\lambda
}\left(  x+X_{\lambda}\right)  =Sx+X_{\lambda}$ defines an operator on
$X/X_{\lambda}$ such that $\hat{S}_{\lambda}\in\mathbb{F}$-OrbRef$\left(
\hat{T}_{\lambda}\right)  .$ Hence, by Theorem \ref{unionker}, since $\hat
{S}_{\lambda}\left(  e+X_{\lambda}\right)  \neq0$, there is a unique $\beta
\in\mathbb{F}$ and a unique nonnegative integer $n$ such that $\hat
{S}_{\lambda}\left(  e+X_{\lambda}\right)  =\beta\hat{T}_{\lambda}^{n}\left(
e+X_{\lambda}\right)  ,$ and for this $\beta$ and $n,$ we have $\hat
{S}_{\lambda}=\beta\hat{T}_{\lambda}^{n}$. Suppose $\eta\in\Lambda$. Since the
$X_{\lambda}$'s are decreasingly directed, there is a $\sigma\in\Lambda$ such
that $X_{\sigma}\subseteq X_{\lambda}\cap X_{\eta}$. Applying the same
arguments we used on $X_{\lambda}$, there is a unique $\alpha\in\mathbb{F}$
and a unique integer $m\geq0$ such that $\hat{S}_{\sigma}\left(  e+X_{\sigma
}\right)  =\alpha T_{\sigma}^{m}\left(  e+X_{\sigma}\right)  .$ However, it
follows that
\[
Se-\alpha T^{n}e\in X_{\sigma}\subseteq X_{\lambda,}%
\]
which implies $\hat{S}_{\lambda}\left(  e+X_{\lambda}\right)  =\alpha\hat
{T}_{\lambda}\left(  e+X_{\lambda}\right)  ,$ which implies that $\alpha
=\beta$ and $m=n.$ Thus $\hat{S}_{\sigma}=\beta\hat{T}_{\sigma}^{n}$, which in
turn implies $\hat{S}_{\eta}=\beta\hat{T}_{\eta}^{n}$. Therefore $\hat
{S}_{\eta}=\beta\hat{T}_{\eta}^{n}$ for every $\eta\in\Lambda$. Therefore, for
every $\eta\in\Lambda$ and for every $x\in X,$%
\[
Sx-\beta T^{n}x\in X_{\eta},
\]
which, by $\left(  2\right)  $, implies $S=\beta T^{n}$.
\end{proof}

\bigskip

The following corollary applies to operators that have a strictly
upper-triangular operator matrix with respect to some direct sum decomposition.

\begin{corollary}
If a normed space $X$ over $\mathbb{F}\in\left\{  \mathbb{R},\mathbb{C}%
\right\}  $ is a direct sum of spaces $\left\{  X_{n}:n\in\mathbb{N}\right\}
$ such that $T\left(  X_{1}\right)  =\left\{  0\right\}  ,$ and for every
$n>1$,
\[
T\left(  X_{n}\right)  \subseteq\left(
{\displaystyle\sum\nolimits_{k<n}^{\oplus}}
X_{k}\right)  ^{-},
\]
then $T$ is $\mathbb{F}$-orbit reflexive and $\mathbb{F}$-Orb$\left(
T\right)  $ is SOT-closed.
\end{corollary}

\bigskip

The preceding corollary has some familiar special cases.\bigskip

\begin{corollary}
If $T$ is an operator-weighted shift or if $T$ is a direct sum of nilpotent
operators on a normed space $X$ over $\mathbb{F}\in\left\{  \mathbb{R}%
,\mathbb{C}\right\}  $, then $T$ is $\mathbb{F}$-orbit reflexive.
\end{corollary}

\bigskip

\begin{theorem}
Suppose $X$ is a normed space over $\mathbb{F}\in\left\{  \mathbb{R}%
,\mathbb{C}\right\}  $, $T\in B\left(  X\right)  $ and $\cap_{n=1}^{\infty
}T^{n}\left(  X\right)  ^{-}=\left\{  0\right\}  $. Then $T$ is $\mathbb{F}%
$-orbit reflexive and $\mathbb{F}$-OrbRef$\left(  T\right)  =\mathbb{F}%
$-Orb$\left(  T\right)  $. Moreover, if $S\in\mathbb{F}$-OrbRef$\left(
T\right)  ,$ $f\in X,$ and $0\neq Sf=\beta T^{k}f$, then $S=\beta T^{k}$.
\end{theorem}

\begin{proof}
We will first show that $T$ is algebraically $\mathbb{F}$-orbit reflexive. If
$M$ is a finite-dimensional invariant subspace for $T$ and $T|M$ is not
nilpotent, then there is a nonzero $T$-invariant subspace $N$ of $M$ such that
$\ker\left(  T|N\right)  =0.$ Thus $T\left(  N\right)  =N\neq0,$ which
violates $\cap_{n=1}^{\infty}T^{n}\left(  X\right)  ^{-}=\left\{  0\right\}
$. Thus, either $T$ is not locally algebraic or $T$ is locally nilpotent. In
these cases it follows either from Theorem \ref{notlocalg} or Theorem
\ref{locnil} that $T$ is indeed algebraically $\mathbb{F}$-orbit reflexive.
Furthermore, the hypothesis on $T$ implies, for each $x\in X,$ that
\[
\cap_{N=1}^{\infty}\left\{  \lambda T^{k}x:\lambda\in\mathbb{F},k\geq
N\right\}  ^{-SOT}=\left\{  0\right\}  \text{,}%
\]
so $\mathbb{F}$-Orb$\left(  T\right)  $ is closed in $X.$ Thus $\mathbb{F}%
$-OrbRef$\left(  T\right)  =\mathbb{F}$-OrbRef$_{0}\left(  T\right)
=\mathbb{F}$-Orb$\left(  T\right)  $. For the last statement suppose $f\in X,$
$\alpha,\beta\in\mathbb{F},$ and $k,n\geq0$ are integers, and
\[
0\neq Sf=\alpha T^{n}f=\beta T^{k}f.
\]
Clearly if $n=k,$ then $\alpha=\beta.$ Suppose $k<n$. Then $M=sp\left\{
f,Tf,\ldots,T^{n-1}f\right\}  $ is a nonzero finite-dimensional invariant
subspace for $T$ with $\dim M\leq n$. Since $T^{n}f\neq0$, we know $T|M$ is
not locally nilpotent, which, as remarked earlier, contradicts $\cap
_{n=1}^{\infty}T^{n}\left(  X\right)  ^{-}=\left\{  0\right\}  $.
\end{proof}

\bigskip

This theorem also implies a stronger version of itself.\bigskip

\begin{corollary}
Suppose $X$ is a normed space over $\mathbb{F}\in\left\{  \mathbb{R}%
,\mathbb{C}\right\}  $, $T\in B\left(  X\right)  ,$ and there is an
increasingly directed family $\left\{  X_{\lambda}:\lambda\in\Lambda\right\}
$ of $T$-invariant linear subspaces such that
\end{corollary}

\begin{enumerate}
\item for every $\lambda\in\Lambda,$ $\cap_{n=1}^{\infty}\overline
{T^{n}\left(  X_{\lambda}\right)  }=\left\{  0\right\}  $, and

\item $\cup_{\lambda\in\Lambda}X_{\lambda}$ is dense in $X.$
\end{enumerate}

Then $T$ is $\mathbb{F}$-orbit reflexive, and $\mathbb{F}$-OrbRef$\left(
T\right)  =\mathbb{F}$-Orb$\left(  T\right)  $. Moreover, if $S\in\mathbb{F}%
$-OrbRef$\left(  T\right)  ,$ $f\in X,$ and $0\neq Sf=\beta T^{k}f$, then
$S=\beta T^{k}$.

\begin{proof}
Suppose $0\neq S\in\mathbb{F}$-OrbRef$\left(  T\right)  $. It follows from
$\left(  2\right)  $ that there is a $\lambda_{0}\in\Lambda$ and an $f\in
X_{\lambda_{0}}$ such that $0\neq Sf.$ However, we must have $S\left(
X_{\lambda_{0}}\right)  \subseteq X_{\lambda_{0}},$ and $S|X_{\lambda_{0}}%
\in\mathbb{F}$-OrbRef$\left(  T|X_{\lambda_{0}}\right)  =\mathbb{F}%
$-Orb$\left(  T\right)  $ (by $\left(  1\right)  $ and the preceding theorem).
Thus there is a unique scalar $\beta$ and an integer $k\geq0$ such that%
\[
S|X_{_{\lambda_{0}}}=\beta T^{k}|X_{\lambda_{0}}.
\]
The same $\beta$ and $k$ must work for any $X_{\lambda}$ that contains
$X_{\lambda_{0}}.$ It follows from the fact that the family is increasingly
directed and $\left(  2\right)  $ that $S=\beta T^{k}$.\bigskip
\end{proof}

I. Kaplansky \cite{Kap} (see also \cite{Lar}, \cite{Lar} , \cite{M}) proved
that a (bounded linear) operator on a Banach space is locally algebraic if and
only if it is algebraic. This immediately gives us the following result from
Corollary \ref{algTrans}.

\begin{proposition}
\label{algebraic}Suppose $X$ is a Banach space over the field $\mathbb{F}%
\in\left\{  \mathbb{R},\mathbb{C}\right\}  $and $T\in B\left(  X\right)  $ is
not algebraic. Then $T$ is algebraically $\mathbb{F}$-orbit reflexive.
\end{proposition}

\ \bigskip

\bigskip If $T$ is an operator on a Banach space, then $r\left(  T\right)  $
denotes the spectral radius of $T$, i.e.,%
\[
r\left(  T\right)  =\max\left\{  \left\vert \lambda\right\vert :\lambda
\in\sigma\left(  T\right)  \right\}  \text{.}%
\]

\begin{theorem}
Suppose $T\in\mathcal{M}_{d}\left(  \mathbb{C}\right)  $ and $T$ is not
nilpotent. The following are equivalent.

\begin{enumerate}
\item $T$ is $\mathbb{C}$-orbit reflexive

\item Among all the Jordan blocks with eigenvalues having modulus equal to
$r\left(  T\right)  >0,$ the two largest blocks differ in size by at most $1$.
\end{enumerate}
\end{theorem}

\begin{proof}
We begin with some basic computations. Suppose $J_{m}$ is an $m\times m$
nilpotent Jordan block, i.e., there is an orthonormal basis $\left\{
e_{0},\ldots,e_{m-1}\right\}  $ for the domain of $J_{m}$ such that
$J_{m}e_{m-1}=0$ and $J_{m}e_{k}=e_{k+1}$ for $0\leq k<m-1$. Note that $J_{m}$
is lower triangular with respect to the basis $\left\{  e_{0},\ldots
,e_{m-1}\right\}  $. Then, for $\alpha\in\mathbb{C}$ with $\left\vert
\alpha\right\vert \leq1$ and $n\geq0$, we have from the binomial theorem that%
\[
\left\Vert \left(  \alpha+J_{m}\right)  ^{n}\right\Vert =\left\Vert \sum
_{k=0}^{n}\binom{n}{k}\alpha^{n-k}J_{m}^{k}\right\Vert =\left\Vert \sum
_{k=0}^{\min(n,m-1)}\binom{n}{k}\alpha^{n-k}J_{m}^{k}\right\Vert \leq
mn^{m-1}\left\vert \alpha\right\vert ^{n-m+1},
\]
so if $\left\vert \alpha\right\vert =1$ and $N>m-1$ or $\left\vert
\alpha\right\vert <1$, then%
\[
\lim_{n\rightarrow\infty}\frac{1}{\binom{n}{N}}\left\Vert \left(  \alpha
+J_{m}\right)  ^{n}\right\Vert =0.
\]
Moreover, for $0\leq t<m-1$, we have%
\[
\left(  \alpha+J_{m}\right)  ^{n}e_{t}=\sum_{k=0}^{m-1}\binom{n}{k}%
\alpha^{n-k}J_{m}^{k}J_{m}^{t}e_{0}=\sum_{k=0}^{m-t-1}\binom{n}{k}\alpha
^{n-k}e_{t+k},
\]
and, when $m\geq2,$ we have%
\[
\left\langle \left(  \alpha+J_{m}\right)  ^{n}e_{0},e_{0}\right\rangle
=\alpha^{n}\text{ and }\left\langle \left(  \alpha+J_{m}\right)  ^{n}%
e_{0},e_{1}\right\rangle =n\alpha^{n-1}%
\]
In particular, if $\left\vert \alpha\right\vert =1$ and $0\leq i<m$ we have%
\[
\lim_{n\rightarrow\infty}\frac{1}{\binom{n}{m-i-1}\alpha^{n-m+i+1}}\left(
\alpha+J_{m}\right)  ^{n}e_{i}=e_{m-1}.
\]

We can assume $T$ is already equal to its Jordan canonical form. By replacing
$T$ with $\frac{1}{\lambda}T,$ where $\lambda$ is an eigenvalue of $T$ with
modulus $r\left(  T\right)  >0$ and having the largest, say $m\times m$,
Jordan block among such eigenvalues, we can assume that this largest block has
eigenvalue $1$. We write $T$ as
\[
T=\left(  1+J_{m}\right)  \oplus%
{\displaystyle\sum\nolimits_{1\leq i\leq s}^{\oplus}}
\left(  \alpha_{i}+J_{m_{i}}\right)  \oplus A
\]
with each $\left\vert \alpha_{i}\right\vert =1,$ and $m\geq m_{1}\geq
\cdots\geq m_{s}$, and the modulus of every eigenvalue of $A$ less than $1$.
It follows that $A^{n}\rightarrow0$ as $n\rightarrow\infty$. Note that we
allow the possibility that $s=0$ or $A$ is not present.

$\left(  1\right)  \Longrightarrow\left(  2\right)  $. Assume $m_{1}\leq m-2$,
i.e., the second largest Jordan block for the eigenvalues with modulus
$r\left(  T\right)  $ differs from $m$ by more than $1$. In this case
$m\geq2.$ Let $\left\{  e_{0},\ldots,e_{m-1}\right\}  $ be the orthonormal
basis above. Define a linear transformation $S$ in terms of the inner product
$\left\langle ,\right\rangle $ on $\mathbb{C}^{n}$ by%
\[
Sx=\left[  \left\langle x,e_{0}\right\rangle +\left\langle x,e_{1}%
\right\rangle \right]  e_{m-1}.
\]
Note that%
\[
\left(  T-1\right)  Se_{0}=J_{m}e_{m-1}=0,
\]
but%
\[
S\left(  T-1\right)  e_{0}=Se_{1}=e_{m-1}\neq0.
\]
Hence $ST\neq TS,$ so $S$ is not in the SOT-closure of $\mathbb{C}%
$-Orb$\left(  T\right)  $. However, we will show that $S\in\mathbb{C}%
$-OrbRef$\left(  T\right)  $. If $x$ is a vector and $Sx=0$, then $Sx=0\cdot
T^{1}x\in\mathbb{C}$-Orb$\left(  T,x\right)  $. If $Sx=\beta e_{m}\neq0$, then
either $\left\langle x,e_{0}\right\rangle \neq0$ or both $\left\langle
x,e_{0}\right\rangle =0$ and $\left\langle x,e_{1}\right\rangle \neq0$. In
case $\left\langle x,e_{0}\right\rangle \neq0$, we have%
\[
Sx=\lim_{n\rightarrow\infty}\frac{\beta}{\left\langle x,e_{0}\right\rangle
}\frac{1}{\binom{n}{m-1}}T^{n}x\in\mathbb{C}\text{-}Orb\left(  T,x\right)
^{-}.
\]
In case $\left\langle x,e_{0}\right\rangle =0$ and $\left\langle
x,e_{1}\right\rangle \neq0,$ we have%
\[
Sx=\lim_{n\rightarrow\infty}\frac{\beta}{\left\langle x,e_{1}\right\rangle
}\frac{1}{\binom{n}{m-2}}T^{n}x\in\mathbb{C}\text{-}Orb\left(  T,x\right)
^{-}.
\]
Both of the above two formulas hold because
\[
\lim_{n\rightarrow\infty}\frac{1}{\binom{n}{m-2}}\left\Vert \left(  \alpha
_{i}+J_{m_{i}}\right)  ^{n}\right\Vert =0
\]
for $1\leq i\leq s$ since either $m_{i}\leq m-2$ and $A^{n}\rightarrow0$.

$\left(  2\right)  \Longrightarrow\left(  1\right)  $. Let $\left\{
f_{0},\ldots,f_{m_{1}-1}\right\}  $ be the orthonormal basis shifted by
$J_{m_{1}}$. Suppose $S\in\mathbb{C}$-OrbRef$\left(  T\right)  $. Relative to
the direct sum decomposition for $T$ above, we can write
\[
S=S_{0}\oplus S_{1}\oplus\cdots\oplus S_{r}\oplus B.
\]

In order to prove $S\in\mathbb{C}$-Orb$\left(  T\right)  ^{-SOT}$, we consider
each of the following cases.

\textbf{Case 1 }$m=1$. Then $T=1\oplus U\oplus A$ with $U$ unitary. Write
$S=\lambda\oplus D\oplus B$.

\textbf{Subcase 1.1} $\lambda=0.$ Suppose $x$ is in the domain of $U\oplus A$.
Then there is a sequence $\left\{  c_{n}\right\}  $ in $\mathbb{C}$ and a
sequence $\left\{  k_{n}\right\}  $ of integers such that
\[
Se_{0}\oplus Sx=S\left(  e_{0}\oplus x\right)  =\lim_{n\rightarrow\infty}%
c_{n}T^{k_{n}}\left(  e_{0}\oplus x\right)  =
\]%
\[
\lim_{n\rightarrow\infty}c_{n}e_{0}\oplus\lim_{n\rightarrow\infty}c_{n}\left(
U\oplus A\right)  ^{k_{n}}x.
\]
Thus $c_{n}\rightarrow0,$ and, since $\left\{  \left\Vert \left(  U\oplus
A\right)  ^{k_{n}}\right\Vert \right\}  $ is bounded, $Sx=0$. Hence, $S=0$.

\textbf{Subcase 1.2} $B=0$, so $S=\lambda\oplus D\oplus0$. It is well-known
that if $\alpha\in\mathbb{C}$ and $\left\vert \alpha\right\vert =1$, then
there is a sequence $\left\{  k_{n}\right\}  \rightarrow\infty$ such that
$\alpha^{k_{n}}\rightarrow1$. Thus there is a sequence $\left\{
k_{n}\right\}  \rightarrow\infty$ such that $U^{k_{n}}\rightarrow1.$ Thus
$T^{k_{n}+1}\rightarrow1\oplus U\oplus0.$ Since $1\oplus U$ is normal,
$\lambda\oplus D\in\mathbb{C}$-Orb$\left(  1\oplus U\right)  ^{-SOT}$. Hence
$S\in\mathbb{C}$-Orb$\left(  T\right)  ^{-SOT}.$

\textbf{Subcase 1.3}. $B\neq0$, $\lambda\neq0.$ Then, for every $x\in DomA$,
there are sequences $\left\{  c_{n}\right\}  $ and $\left\{  k_{n}\right\}  $
such that
\[
\lambda e_{0}\oplus0\oplus Bx=S\left(  e_{0}\oplus0\oplus x\right)  =\lim
c_{n}T^{k_{n}}\left(  e_{0}\oplus0\oplus x\right)  ,
\]
which implies$,$ $c_{n}\rightarrow\lambda$. By choosing a subsequence, we can
assume $k_{n}\rightarrow\infty$ or $k_{n}\rightarrow k<\infty$. If
$k_{n}\rightarrow\infty$, then $Bx=0.$ If $k_{n}\rightarrow k$, then
$Bx=\lambda A^{k}x.$ Hence
\[
Dom\left(  B\right)  =\ker B\cup\bigcup_{n=0}^{\infty}\ker\left(  B-\lambda
A^{n}\right)  ,
\]
which implies, by the Baire Category theorem, that $B\in\left\{
0,\lambda,\lambda A,\lambda A^{2},\ldots\right\}  $. Since $B\neq0$, there is
a $k\geq0$ such that $B=\lambda A^{k}$. Since $1-A^{t}$ is invertible for
$t\geq1$, and $A^{k}=\frac{1}{\lambda}B\neq0$, the integer $k$ is unique.
Applying the same technique with $e_{0}$ replaced with an eigenvector for $U,$
we get $D=\lambda U^{k}$. Thus, $S=\lambda T^{k}$.

\textbf{Case 2} $m\geq2$. Then $m_{1}\geq m-1\geq1$. Write $S=S_{0}\oplus
S_{1}\oplus\cdots\oplus S_{s}\oplus B$.

\textbf{Subcase 2.1} $B\neq0$. As in Subcase 1.3, $S=\lambda T^{k}$.

\textbf{Subcase 2.2}. $Se_{0}=0.$ Suppose $x$ is orthogonal to the domain of
$S_{0}$. Then
\[
S\left(  e_{0}\oplus x\right)  =\lim_{n\rightarrow\infty}c_{n}T^{k_{n}}%
e_{0}\oplus c_{n}T^{k_{n}}e_{0}x
\]
implies that
\[
\lim_{n\rightarrow\infty}c_{n}\binom{k_{n}}{m-1}=0,
\]
which implies
\[
\lim_{n\rightarrow\infty}c_{n}\left\Vert \left(  \left(  1+J_{m}\right)
\oplus\left(  \alpha_{1}+J_{m_{1}}\right)  \oplus\cdots\oplus\left(
\alpha_{r}+J_{m_{r}}\right)  \oplus A\right)  ^{k_{n}}\right\Vert
\rightarrow0.
\]
This, in turn, implies $Sx=0.$ Thus $S_{1}\oplus\cdots\oplus S_{r}\oplus B=0$.
Applying the same idea with $e_{0}$ replaced with $f_{0}$ and $x$ replaced
with any of $e_{1},\ldots,e_{m-1}$, we conclude that $S_{0}=0.$ Hence $S=0$.

\textbf{Subcase 2.3} $Se_{0}=\lambda T^{k}e_{0}\neq0$. Suppose $x$ is
orthogonal to the domain of $S_{0}$ and
\[
S\left(  e_{0}+x\right)  =\lim_{n\rightarrow\infty}c_{n}T^{k_{n}}\left(
e_{0}+x\right)  .
\]
We then get
\[
\lim_{n\rightarrow\infty}c_{n}=\lim_{n\rightarrow\infty}\left\langle
c_{n}T^{k_{n}}e_{0},e_{0}\right\rangle =\left\langle Se_{0},e_{0}\right\rangle
=\lambda,
\]
and%
\[
\lim_{n\rightarrow\infty}k_{n}=\lim_{n\rightarrow\infty}\frac{1}{c_{n}%
}\left\langle c_{n}T^{k_{n}}e_{0},e_{1}\right\rangle =\left\langle \frac
{1}{\lambda}Se_{0},e_{1}\right\rangle =k.
\]
Hence $Sx=\lambda T^{k}x$ for every $x$ orthogonal to the domain of $S_{0}$.
In particular%
\[
S_{1}=\lambda\left(  \alpha_{1}+J_{m_{1}}\right)  ^{k}.
\]
If $m_{1}\geq2$, we can make the same argument with $e_{0}$ replaced with
$f_{0}$ to get that $S_{0}=\lambda\left(  1+J_{m}\right)  ^{k}$, implying
$S=\lambda T^{k}$. If $m_{1}=1,$ then $m=2,$ and we need only show that
$Se_{1}=\lambda T^{k}e_{1};$ but we know that $Te_{1}=e_{1}$. Suppose
$Se_{1}=\lambda\beta e_{1}$ and $\beta\neq1.$ Choose $y=\frac{1}{2\left(
\beta-1\right)  }$; we can write%
\[
\lambda e_{0}+\left(  \lambda k+\lambda\beta y\right)  e_{1}=S\left(
e_{0}+ye_{1}\right)  =\lim_{n\rightarrow\infty}d_{n}T^{j_{n}}\left(
e_{0}+ye_{1}\right)  =\lim_{n\rightarrow\infty}d_{n}e_{0}+\left(  d_{n}%
j_{n}+y\right)  e_{1},
\]
which implies that
\[
d_{n}\rightarrow\lambda,
\]
and
\[
j_{n}-k\rightarrow\left(  \beta-1\right)  y=\frac{1}{2},
\]
which is impossible. Thus $Se_{1}=\lambda e_{1},$ and $S=\lambda T^{k}$.

\textbf{Subcase 2.4} $Se_{0}=\lim c_{n}T^{k_{n}}e_{0}\neq0$ with
$n_{k}\rightarrow\infty$. Using the ideas in the proof of Subcase 2.3, we get
$B=0,$ and $S_{j}=0$ when $m_{j}<m,$ and $S_{j}|ran\left(  J_{m_{j}}\right)
=0$ when $m_{j}=m$. If $m_{1}=m,$ we can apply the same reasoning to
$e_{i}\oplus f_{0}$ for $1\leq i<m$ to get $Se_{i}=0,$ and then applying $S$
to a sum of $e_{0}\oplus f_{0}\oplus h_{2}\oplus\cdots\oplus h_{t}\oplus0$,
where $m_{j}=m$ and $h_{j}\in\ker J_{m_{j}}^{\ast}$ for $1\leq j\leq t$, we
conclude that there are sequences $\left\{  d_{n}\right\}  $ and $\left\{
j_{n}\right\}  $ such that $S=\lim_{n\rightarrow\infty}d_{n}T^{j_{n}}$ and
$\lim_{n\rightarrow\infty}j_{n}=\infty$. If $m_{1}=m-1,$ then $Sf_{0}=0,$ and
we can still look at $S\left(  e_{i}\oplus f_{0}\right)  $ for $1\leq i<m$ to
get $Se_{i}=0$. In this case we get $S=\lim_{n\rightarrow\infty}c_{n}T^{k_{n}%
}$.

Hence in all of the possible cases, $S\in\mathbb{C}$-Orb$\left(  T\right)
^{-SOT}.$ Thus $T$ is $\mathbb{C}$-orbit reflexive.
\end{proof}

\bigskip\bigskip

\section{Orbit Reflexivity\bigskip}

We conclude with a few results on orbit reflexivity, most of which appeared in
\cite{McH}. A key ingredient in the results of this section comes from
\cite[Theorem 5 (1)]{HNRR}, which uses a simple Baire category argument to
show, for an operator $T$ on a Banach space, that if $Orb\left(  T,x\right)  $
is closed for every $x$ in a nonempty open set, then $T$ is orbit reflexive.

\begin{lemma}
\label{fhs} Suppose $X$ is a normed space, $T\in B\left(  X\right)  ,$
$\lambda$ is an eigenvalue of the adjoint $T^{\#}$ of $T$ with unit
eigenvector $\alpha\in X^{\#}$ and $\left\vert \lambda\right\vert >1$. Then
$T$ is orbit reflexive.
\end{lemma}

\begin{proof}
Suppose $f\in X$ and $\left\langle f,\alpha\right\rangle =\alpha\left(
f\right)  \neq0.$ Then
\[
\left\Vert T^{n}f\right\Vert \geq\left\vert \left\langle T^{n}f,\alpha
\right\rangle \right\vert =\left\vert \left\langle f,\left(  T^{n}\right)
^{\ast}\alpha\right\rangle \right\vert =\left\vert \lambda\right\vert
^{n}\left\vert \left\langle f,\alpha\right\rangle \right\vert \rightarrow
\infty
\]
as $n\rightarrow\infty$. Since $\left\{  f\in X:\left\langle f,\alpha
\right\rangle \neq0\right\}  $ is an open set, it follows from \cite[Theorem 5
(1)]{HNRR} that $T$ is orbit reflexive.
\end{proof}

\begin{corollary}
Suppose $X$ is a functional Hilbert space on a set $E$, and $f:E\rightarrow
\mathbb{C}$ is a multiplier of $X$ such that $\left\Vert M_{f}\right\Vert
=\sup\left\{  \left\vert f\left(  t\right)  \right\vert :t\in E\right\}  $.
Then $M_{f}$ is orbit-reflexive.
\end{corollary}

\begin{proof}
Suppose $t\in E$ and let $e_{t}\in X^{\#}$ be the evaluation functional at
$t.$ Then, for every $h\in X,$ we have%
\[
\left(  M_{f}^{\#}\left(  e_{t}\right)  \right)  h=e_{t}\left(  M_{f}h\right)
=e_{t}\left(  fh\right)  =f\left(  t\right)  e_{t}\left(  h\right)  .
\]
Thus $M_{f}^{\#}e_{t}=f\left(  t\right)  e_{t}$ for every $t\in E$. If
$\left\vert f\left(  t\right)  \right\vert >1$ for some $t\in E$, then it
follows from Lemma \ref{fhs} that $M_{f}$ is orbit reflexive. Otherwise,
$\left\Vert M_{f}\right\Vert \leq1,$ which, by \cite{HNRR}, implies $M_{f}$ is
orbit reflexive.
\end{proof}

\bigskip

In \cite{McH} the third author used the preceding corollary and a result of J.
E. Thomson \cite{T} concerning bounded point evaluations for cyclic subnormal
operators to show that every cyclic subnormal operator is a multiplication on
a functional Hilbert space, implying that every cyclic subnormal operator is
orbit reflexive. Here we give a more elementary proof that every subnormal
operator is orbit reflexive.

\bigskip

\begin{theorem}
Suppose $H$ is a Hilbert space and $T\in B\left(  H\right)  $ is subnormal.
Then $T$ is orbit reflexive.
\end{theorem}

\begin{proof}
Suppose $f\in H$ and $\left\Vert f\right\Vert =1$. Let $E_{f}=\left\{
p\left(  T\right)  f:p\in\mathbb{C}\left[  t\right]  \right\}  ^{-}$ be the
cyclic invariant subspace for $T$ generated by $f$. We know (see \cite{C})
that there is a probability measure $\mu$ whose support is $\sigma\left(
T|_{E_{f}}\right)  $ and a unitary operator $U$ from $E_{f}$ onto the closure
$P^{2}\left(  \mu\right)  $ of the set of polynomials in $L^{2}\left(
\mu\right)  $ such that $Uf=1$ and $UT|_{E_{f}}U^{\ast}$ is the multiplication
operator $M_{z}$ on $P^{2}\left(  \mu\right)  $. Since the norm of a subnormal
operator equals it spectral radius and
\[
\left\Vert T^{n}f\right\Vert ^{2}=\int_{\sigma\left(  T|_{E_{f}}\right)
}\left\vert z^{2n}\right\vert d\mu
\]
for each $n\geq1$, we have that
\[
\left\{  T^{n}f\right\}  \text{ is bounded}\Leftrightarrow\left\Vert
T|_{E_{f}}\right\Vert \leq1\Leftrightarrow sup_{n\geq1}\left\Vert
T^{n}f\right\Vert \leq\left\Vert f\right\Vert ,
\]
and
\[
\left\Vert T|_{E_{f}}\right\Vert >1\Leftrightarrow\lim_{n\rightarrow\infty
}\left\Vert T^{n}f\right\Vert =\infty.
\]
It follows that $\left\{  f\in H:\left\{  T^{n}f\right\}  \text{ is
bounded}\right\}  $ is closed and the set
\[
U=\left\{  f\in H:\lim_{n\rightarrow\infty}\left\Vert T^{n}f\right\Vert
=\infty\right\}
\]
is open. If $U=\varnothing$, then $\left\Vert T\right\Vert \leq1,$ which
implies $T$ is orbit reflexive \cite{HNRR}. On the other hand if
$U\neq\varnothing$, then it follows from \cite[Theorem 5 (1)]{HNRR} that $T$
is orbit reflexive.
\end{proof}

\end{document}